\documentclass[12pt,a4paper,twoside]{article}

\usepackage{fancyhdr}
\usepackage[a4paper,right=2cm,left=2cm,top=2.8cm,bottom=2cm]{geometry}
\newcommand{\nmpcfoot}{Draft paper, 2011}
\newcommand{\Atitle}[1]{
{\Large\bf
\begin{center}
#1
\end{center}
\vspace{.5cm}
}}
\newcommand{\Aauthor}[1]{
{\large
\begin{center}
#1
\end{center}
\vspace{.0cm}
}}
\newcommand{\Aaddress}[1]{
{\em
\begin{center}
#1
\end{center}
\vspace{.2cm}
}}
\newcommand{\Aabstract}[1]{
\noindent{\bf Abstract:} #1
\vspace{.5cm}
}
\newcommand{\Akeywords}[1]{
\noindent{\bf Keywords} : #1
\vspace{.5cm}
}
\usepackage{amssymb}
\usepackage{amsmath}
\usepackage{amsthm}
\usepackage{amscd}
\usepackage{amsfonts}
\usepackage{graphics}
\usepackage{epsfig}
\usepackage{color}
\usepackage{cite}

\newtheorem{thm}{Theorem}[section]

\newtheoremstyle{style1}{\topsep}{\topsep}
{\rmfamily}	
{}		
{\bfseries}	
{}		
{\newline}	
{}		

\newtheoremstyle{style2}{\topsep}{\topsep}
{\itshape}	
{}		
{\bfseries}	
{}		
{\newline}	
{}		

\theoremstyle{style1}
\newtheorem{definition}[thm]{Definition}
\newtheorem{example}[thm]{Example}
\newtheorem{algorithm}[thm]{Algorithm}

\newtheorem{assumption}[thm]{Assumption}

\theoremstyle{style2}
\newtheorem{theorem}[thm]{Theorem}

\newtheorem{corollary}[thm]{Corollary}
\newtheorem{proposition}[thm]{Proposition}
\newtheorem{remark}[thm]{Remark}

\def\R{\mathbb{R}}
\def\N{\mathbb{N}}

\def\X{X}
\def\bX{\mathbb{X}}
\def\U{U}
\def\bU{\mathbb{U}}

\def\cB{\mathcal{B}}
\def\cK{\mathcal{K}}

\def\cKL{\mathcal{K\hspace*{-0mm}L}}

\def\cS{\mathcal{S}}

\def\l{\ell}

\def\xs{x^\star}

\def\us{u^\star}

\def\argmin{\mathop{\rm argmin}}
\usepackage{framed}
\newenvironment{fshaded}{%
\MakeFramed {\FrameRestore}}%
{\endMakeFramed}
\definecolor{shadecolor}{rgb}{1.,1.,1.}%
\definecolor{framecolor}{rgb}{.0,.0,.0}%
\begin{document}
\Atitle{Stability and Performance Guarantees for MPC Algorithms without Terminal Constraints\footnote{This work was supported by the DFG Priority Project 1305 ``Control Theory of Digitally Networked Dynamical Systems'', grant number Gr1569/12, and the Leopoldina Fellowship Programme LPDS 2009-36.}}

\Aauthor{J\"{u}rgen Pannek$^\dagger$\footnote{Corresponding author,~e-mail:~\textsf{juergen.pannek@googlemail.com}, Phone: +49\,89\,6004\,2567, Fax: +49\,89\,6004\,2136} and Karl Worthmann$^\ddagger$}

\Aaddress{%
$^\dagger$ University of the Federal Armed Forces, 85577 Munich, Germany, {\tt
juergen.pannek@googlemail.com}\\
$^\ddagger$ TU Ilmenau, 98693 Ilmenau, Germany, {\tt 
karl.worthmann@tu-ilmenau.de}
}

\Akeywords{Asymptotic stability, feedback, model predictive control algorithm, performance, receding horizon control.}

\Aabstract{A typical bottleneck of model predictive control algorithms is the computational burden in order to compute the receding horizon feedback law which is predominantly determined by the length of the prediction horizon. Based on a relaxed Lyapunov inequality we present techniques which allow us to show stability and suboptimality estimates for a reduced prediction horizon. In particular, the known structural properties of suboptimality estimates based on a controllability condition are used to cut the gap between theoretic stability results and numerical observations.}

\newcommand{\rauthor}{J.Pannek, K.Worthmann}

\section{Introduction}

\section{Introduction}

Model predictive control (MPC), also termed receding horizon control, is a well established method for the optimal control of linear and nonlinear systems \cite{CB2004, RM2009} and also widely used in industry
, cf. \cite{QB2003, GSAFBD2010}. The method generates a sequence of finite horizon optimal control problems in order to approximate the solution of an infinite horizon optimal control problem, the latter being, in general, computationally intractable. Each solution of such a finite horizon control problem yields a 
sequence of control values of which only the first element is implemented at the corresponding time step. This paradigm is 
iteratively applicable and generates a closed loop static state feedback. Due to the replacement of the original infinite by a sequence of finite horizon control problems, stability of the resulting closed loop may be lost. 

To ensure stability, several modifications of the finite horizon control problems such as stabilizing terminal constraints \cite{KG1988} or a local Lyapunov function \cite{ChAl1998} as an additional terminal weight have been proposed. Since the construction of suitably chosen terminal costs, e.g. a local control Lyapunov function, is a challenging task, within this work, we focus on the computationally attractive form of MPC methodology without stabilizing terminal constraints and/or costs. For such a control loop feasibility, stability, and suboptimality has been shown in \cite{TunaMessinaTeel2006,GrueneNMPC2012} via a relaxed Lyapunov inequality, see also \cite{SX1997,PrimbsNevistic2000} for the linear case. However, since either a controllability condition or knowledge on the optimal value function of the finite horizon subproblems is required in order to deduce concise bounds on the prediction horizon length, the computation of an appropriate receding horizon invariant initial set satisfying these assumption is also demanding --- in particular for nonlinear systems. 
Furthermore, the obtained results, e.g. performance bounds, may be conservative since stable behavior of the MPC controlled closed loop can be observed in many examples although stability can be guaranteed by means of these theoretical results for a longer prediction horizon only, cf. \cite{So1992}. 

Here, our main goal is to reduce this conservatism with respect to the estimated prediction horizon length which is of particular interest since the computational burden of the MPC algorithm grows rapidly with the prediction horizon. Hence, we want to design an MPC algorithm which ensures desired stability properties at runtime for a given initial condition for small prediction horizons. To this end, we also employ a relaxed Lyapunov inequality but do not aim at verifying it at each sampling instant. Our starting point is a result shown in \cite{GPSW2010}: implementing more than only the first element of the computed open loop sequence of control values and checking the relaxed Lyapunov inequality at the next update time enhances the suboptimality bound from \cite{GrueneNMPC2012} for a large class of control systems. 
While the idea of utilizing several open loop elements has already been discussed in \cite{NS1994}, doing so may be harmful in terms of robustness, cf. \cite{MS2007}. To cover this issue, conditions are presentedwhose satisfaction guarantee that the stability condition is maintained and the control loop is closed more often.

Additionally, we deduce stability and performance estimates which allow us to violate the relaxed Lyapunov inequality temporarily, much alike the watchdog technique used in nonlinear optimization \cite{CPLP1982}. To this end, we use an idea similar to \cite{G2010} and monitor the progress made in the decrease of the value function along the closed loop. Here, we point out that these conditions are checkable at runtime of the algorithm which 
renders such an approach applicable in practice.

The paper is organized as follows: In the following section the considered MPC problem and a trajectory based stability theorem from \cite{GP2011} are presented. Section \ref{Section:Reducing the optimization horizon using longer control horizons} generalizes this result to an MPC setting which allows for implementing longer parts of the sequence of control values computed at an update time instant. Furthermore, a first algorithm utilizing this weaker condition is derived. To carry over the improvements of the suboptimality bound, an algorithmic approach is presented in Section \ref{Section:Robustification by reducing the control horizon}.
The issue of an exit strategy is addressed in Section \ref{Section:Acceptable violations of relaxed Lyapunov inequality} which allows us to temporarily violate the relaxed Lyapunov inequality and to further improve previously derived suboptimality estimates. Last, conclusions are drawn in Section \ref{Section:Conclusions}.
Throughout the work a simple example repeatedly provides insight into the improvements achieved by the proposed results.

\section{Problem Setup and Preliminaries}
\label{Section:Setup and Preliminaries}

Let $\N_0$ denote the natural numbers including zero and $\R_0^+$ the nonnegative real numbers. A continuous function $\gamma: \R_0^+ \rightarrow \R_0^+$ is said to be of class $\cK_\infty$ if it satisfies $\gamma(0) = 0$, is strictly increasing and unbounded. Furthermore, a continuous function $\beta: \R_0^+ \times \R_0^+ \rightarrow \R_0^+$ is of class $\cKL$ if it is strictly decreasing in its second argument with $\lim_{t \rightarrow \infty} \beta(r, t) = 0$ for each $r > 0$ and satisfies $\beta(\cdot, t) \in \cK_\infty$ for each $t \geq 0$. 

Let $X$ and $U$ be arbitrary Banach spaces equipped with the metrics $d_{X}: X \times X \to \R_0^+$ and $d_{U}: U \times U \to \R_0^+$. In this work nonlinear discrete time control systems of the form
\begin{align}
	\label{Setup:nonlinear discrete time system}
	x(n + 1) = f(x(n), u(n)), \quad x(0) = x_{0}
\end{align}
for $n \in \N_0$ are considered. $x(n) \in \bX \subset \X$ and $u(n) \in \bU \subset \U$ represent the state and control of the system at time instant $n$. Hence, $\X$ and $\U$ stand for the state and the control value space, respectively. Since $\X$ and $\U$ are arbitrary metric spaces, the following results are applicable to discrete time dynamics induced by a sampled -- finite or infinite dimensional -- systems, see, e.g., \cite{AGW2010a, IK2002}. Here, using subsets $\bX$ and $\bU$ allows for incorporating constraints. Throughout this work, the space of control sequences $u: \N_0 \rightarrow \bU$ is denoted by $\bU^{\N_0}$.

For control system \eqref{Setup:nonlinear discrete time system} we want to construct a static state feedback $u = \mu(x)$ which asymptotically stabilizes the system at a desired equilibrium $\xs$ via model predictive control. The trajectory generated by the feedback map $\mu: \bX \rightarrow \bU$ is denoted by $x_\mu(n) = x_\mu(n;x_0)$, $n=0,1,\ldots$. In order to define asymptotic stability rigorously, the open ball with center $x$ and radius $r$ is denoted by $\cB_r(x)$ and the distance from $x$ to the equilibrium $\xs$ is represented by $\| x \|_{\xs} := d_{\X}(x, \xs)$. 

\begin{definition}\label{Setup:def:stability}
	Let a static state feedback $\mu: \bX \rightarrow \bU$ satisfying $f(\xs,\mu(\xs)) = \xs$ for system \eqref{Setup:nonlinear discrete time system} be given, i.e., $\xs$ is an equilibrium for the closed loop. Then, $\xs$ is said to be (locally) {\em asymptotically stable} if there exist $r>0$ and a function $\beta(\cdot,\cdot) \in \cKL$ such that, for each $x_0 \in \cB_r(\xs) \cap \bX$,
	\begin{align}\label{Setup:def:stability:eq1}
		x_\mu(n;x_0) \in \bX \qquad\text{and}\qquad \| x_\mu(n;x_0) \|_{\xs} \leq \beta( \| x_0 \|_{\xs}, n) \quad\forall\, n \in \N_0.
	\end{align}
\end{definition}

Without loss of generality existence of a control value $\us \in \bU$ such that $f(\xs, \us) = \xs$ holds is assumed throughout this work. The quality of the feedback is evaluated in terms of the infinite horizon cost functional
\begin{align}\label{Setup:infinite cost functional}
	J_{\infty} (x_{0}, u) = \sum\limits_{n = 0}^\infty \l(x(n), u(n))
\end{align}
with continuous stage cost $\l: \X \times \U \rightarrow \R_0^+$ satisfying $\l(\xs, \us) = 0$ and $\l(x,u) > 0$ for all $u \in \U$ for each $x \neq \xs$.  The optimal value function corresponding to \eqref{Setup:infinite cost functional} is denoted by $V_\infty(x_0) = \inf_{\{u \in \bU^{\N_0}: x(n) \in \bX\, \forall n \in \N\}} J_\infty(x_0, u)$. Now, utilizing Bellman's principle of optimality for the optimal value function $V_\infty(\cdot)$ yields the Lyapunov equation
\begin{align}\label{Setup:eq:Lyapunov equation}
	V_\infty(x_0) = \inf_{\{u \in \bU: f(x_0,u) \in \bX\}} \{\l(x_0,u) + V_\infty(f(x_0,u)) \}.
\end{align}
In order to avoid technical difficulties, let the infimum of this equality with respect to $u$ be attained, i.e., there exists a feasible control value $u_{x_0}^\star$, i.e. $u_{x_0}^\star \in \bU: f(x_0,u_{x_0}^\star) \in \bX$, such that $V_\infty(x_0) = \l(x_0,u_{x_0}^\star) + V_\infty(f(x_0,u_{x_0}^\star))$ holds. $u_{x_0}^\star$ is called the $\argmin$ of the right hand side of \eqref{Setup:eq:Lyapunov equation} --- albeit uniqueness of the minimizer $u_{x_0}^\star$ is not required. In case of uniqueness the $\argmin$-operator can be understood as an assignment, otherwise it is just a convenient way of writing ``$u_{x_0}^\star$ minimizes the right hand side of \eqref{Setup:eq:Lyapunov equation}''. Hence, an optimal feedback law on the infinite horizon is defined via
\begin{align}\label{Setup:infinite control}
	\mu_{\infty}(x(n)) := \argmin_{\{u \in \bU: f(x(n),u) \in \bX\}} \left\{ \l(x(n), u) + V_{\infty}(f(x(n),u)) \right\} 
\end{align}
In general, the computation of the control law \eqref{Setup:infinite control} requires the solution of a Bellman equation which is, in general, very hard to solve. In order to avoid this burden, we utilize a model predictive control approach without terminal constraints and/or costs to approximate the infinite horizon optimal control. Here, we like to point out that, if a terminal cost can be constructed such that the ``cost to go'' is sufficiently well approximated as in quasi infinite horizon optimal control, see, e.g. \cite{ChAl1998}, its incorporation can contribute to the controller performance. However, if the domain of the terminal cost is too small, then the needed coupling terminal constraint may lead to drawbacks from a performance point of view.

MPC consists of three distinct steps: First, given measurements of the current state $x_{0}$, an optimal sequence of control values over a truncated and, thus, finite horizon is computed which minimizes the cost functional
\begin{align*}
	J_{N}(x_{0}, u) = \sum_{k = 0}^{N - 1} \l(x_{u}(k; x_{0}), u(k)).
\end{align*}
Here, $x_u(\cdot;x_0)$ denotes the trajectory emanating from $x_0$ corresponding to the input signal $u(\cdot)$. As a result, we obtain the open loop control $u_{N}(\cdot; x_{0}) = \argmin_{\{u \in \bU^N: f(x_{0},u) \in \bX\}} J_{N}(x_{0}, u)$ and the corresponding open loop state trajectory
\begin{align}
	\label{Setup:open loop system}
	x_{u_{N}}(n + 1; x_{0}) = f(x_{u_{N}}(n; x_{0}), u_{N}(n; x_{0}))
\end{align}
for $n = 0, \ldots, N - 1$ with initial value $x_{u_{N}}(0; x_{0}) = x_{0}$. In a second step, the first element of the open loop control is employed in order to define the feedback law $\mu_{N}(x_{0}) := u_{N}(0; x_{0})$. Last, the feedback is applied to the system under control revealing the closed loop system
\begin{align}
	\label{Setup:closed loop system}
	x_{\mu_{N}}(n + 1; x_{0}) = f(x_{\mu_{N}}(n; x_{0}), \mu_{N}(x_{\mu_{N}}(n; x_{0})))
\end{align}
for which the abbreviation $x_{\mu_{N}}(\cdot)$ will be used whenever the initial value $x_{0}$ is clear. The closed loop costs with respect to the feedback law $\mu_N(\cdot)$ are given by $V_\infty^{\mu_N}(x_0) = \sum_{n = 0}^\infty \l(x_{\mu_{N}}(n), \mu_{N}(x_{\mu_{N}}(n)))$ and $V_N(x_0)$ 
denotes the optimal value function corresponding to $J_N(x_0, \cdot)$.

In the literature endpoint constraints or a Lyapunov function type endpoint weight are used to ensure stability of the closed loop, see, e.g., \cite{KG1988, ChAl1998, JH2005, Graichen2010}. Here, we focus on an MPC version without these modifications. Hence, feasibility of the MPC scheme is an issue that cannot be neglected. In order to exclude a scenario in which the closed loop trajectory runs into a dead end, the following viability condition is assumed.
\begin{assumption}
	For each $x \in \bX$ there exists a control value $u_{x} \in \bU$ satisfying $f(x, u_{x}) \in \bX$.
\end{assumption}
We like to stress that the computation of such a control invariant set $\bX$ is a hard task. However, suitable methods exists, see, e.g. \cite{KerriganMaciejowski2000}. Furthermore, we refer to \cite{PrimbsNevistic2000}, \cite{GrueneNMPC2012} for a detailed discussion on feasibility issues in MPC without terminal constraints and/or costs. 

In order to be applicable at runtime of the algorithm, we present our conditions in a trajectory based setting, i.e. contrary to Definition \ref{Setup:def:stability} we suppose conditions to be satisfied along a particular closed loop trajectory only. While such conditions are less restrictive, we stress this difference by calling the closed loop solution to behave like a solution of an asymptotically stable system. According to \cite[Proposition 7.6]{GP2011}, stable behavior of the closed loop trajectory can be guaranteed for such a controller using relaxed dynamic programming, cf. \cite{LR2006}.
\begin{theorem}\label{Setup:thm:aposteriori}
	(i) Consider the feedback law $\mu_{N}: \bX \rightarrow \bU$ and the closed loop trajectory $x_{\mu_{N}}(\cdot)$ of \eqref{Setup:closed loop system} with initial value $x_0 \in \bX$. If the value function $V_{N}: \bX \rightarrow \R_{0}^{+}$ satisfies 
	\begin{align}
		\label{Setup:thm:aposteriori:eq1}
		V_{N}(x_{\mu_{N}}(n)) \geq V_{N}(x_{\mu_{N}}(n + 1)) + \alpha \l(x_{\mu_{N}}(n), \mu_{N}(x_{\mu_{N}}(n))) \qquad\forall\, n \in \N_0
	\end{align}
	for some $\alpha \in (0, 1]$, then the performance estimate 
	\begin{align}
		\label{Setup:thm:aposteriori:eq2}
		\alpha V_{\infty}(x_{\mu_{N}}(n)) \leq \alpha V_{\infty}^{\mu_N}(x_{\mu_{N}}(n)) \leq V_{N}(x_{\mu_{N}}(n)) \leq V_{\infty}(x_{\mu_{N}}(n)) 
	\end{align}
	holds for all $n \in \N_{0}$.\\
	(ii) If, in addition, there exist $\alpha_{1}, \alpha_{2}, \alpha_{3} \in \cK_\infty$ such that $\alpha_{1}( \| x \|_{\xs} ) \leq V_{N}(x) \leq \alpha_{2}( \| x \|_{\xs} )$ and $\l(x,u) \geq \alpha_{3}( \| x \|_{\xs} )$ hold for $x = x_{\mu_N}(n) \in \bX$, $n \in \N_{0}$, then there exists $\beta\in \cKL$ which only depends on $\alpha_{1}, \alpha_{2}, \alpha_{3}$ and $\alpha$ such that \eqref{Setup:def:stability:eq1} holds, i.e., $x_{\mu_{N}}(\cdot)$ behaves like a trajectory of an asymptotically stable system.
\end{theorem}
Within Theorem \ref{Setup:thm:aposteriori} the relaxed Lyapunov Inequality \eqref{Setup:thm:aposteriori:eq1} is the key assumption. From the literature, it is well--known that this condition is satisfied for sufficiently long horizons $N$, cf. \cite{JH2005, GMTT2005, AB1995}, and that a suitable $N$ may be computed via methods described in \cite{NP1997,TunaMessinaTeel2006,GPSW2010}. The suboptimality index $\alpha$ in this inequality can be interpreted as a lower bound for the rate of convergence. 

\begin{remark}[Performance measure]
	Inequality \eqref{Setup:thm:aposteriori:eq2} yields a performance bound for the MPC closed loop in comparison to the optimal costs on the infinite horizon. This assessment of the achieved closed loop performance is particularly appealing for $\alpha \notin \{0,1\}$ since MPC with terminal cost yields either (approximately) optimal performance, i.e. $\alpha \approx 1$ or solely stability $\alpha > 0$ without rigorous bounds on the degree of suboptimality.
\end{remark}

\section{Shortening the Prediction Horizon by Weakening the Relaxed Lyapunov Inequality}
\label{Section:Reducing the optimization horizon using longer control horizons}

For the described MPC setting we want to guarantee stability and a certain lower bound on the degree of suboptimality with respect to the infinite horizon optimal control law $\mu_\infty(\cdot)$. Yet, at the same time the optimization horizon $N$ shall be as small as possible. These goals oppose each other since it is known from the literature, see, e.g., \cite{JH2005, GMTT2005, AB1995}, that stability can only be guaranteed if the optimization horizon is sufficiently long. Keeping $N$ small, however, is important from a practical point of view since the horizon length is the dominating factor regarding the computational burden. Here, motivated by theoretical results deduced in \cite{GPSW2010}, our goal is to improve the stability and suboptimality bounds from \cite[Proposition 7.6]{GP2011} and, thereby, to allow the use of smaller prediction horizons.

Note that Condition \eqref{Setup:thm:aposteriori:eq1} is a sufficient, yet not a necessary condition. In particular, even if stability and the desired performance Estimate \eqref{Setup:thm:aposteriori:eq2} cannot be guaranteed via Theorem \ref{Setup:thm:aposteriori}, stable and satisfactory behavior of the closed loop may still be observed, even in the linear case, as illustrated by the following example.
\begin{example}\label{Reducing the optimization horizon using longer control horizons:ex1}
	Let $\bX = X = \R^2$, $\bU = U = \R$ be given and consider the linear control system $x(n+1) = Ax(n) + Bu(n)$ with quadratic stage cost function $\l(x, u) := x^\top Q x + u^\top R u$ from \cite{SX1997} with
	\begin{align*}
		A = \begin{pmatrix}
			  1 & 1.1 \\ -1.1 & 1
		    \end{pmatrix},
		\qquad
		B = \begin{pmatrix}
			  0 \\ 1
		    \end{pmatrix},
		\qquad
		Q = \begin{pmatrix}
			  1 & 0 \\ 0 & 1
		    \end{pmatrix},
		\quad \text{and} \quad
		R = \begin{pmatrix}
			  1
		    \end{pmatrix}.
	\end{align*}
	For the given system, the optimal finite horizon costs $V_N(x_0) = x_0^T P_N x_0$ be can be computed via the Riccati difference equation using the initial condition $P_1 = Q$ according to
	\begin{align*}
		 P_{j+1} = A^\top \left[ P_j - P_j B \left( B^\top P_j B + R \right)^{-1} B^\top P_j \right] A + Q.
	\end{align*}
	Then, the resulting open loop control law $u_N(\cdot, x)$ is given by
	\begin{align*}
		u_N(k; x) = - \left( B^\top P_{N - k} B + R \right)^{-1} B^\top P_{N - k} A x_{u_N}(k; x),
	\end{align*}
	$k=0,1,\ldots,N-2$, and $u_N(N-1;x) = 0$, cf. \cite{NP1997}. Now, we implement the Riccati feedback in the usual receding horizon fashion $\mu_N(x(n)) := u_N(0; x(n))$ with $N = 3$ and evaluate the suboptimality bound from Theorem \ref{Setup:thm:aposteriori} along the closed loop. As displayed in Figure \ref{Reducing the optimization horizon using longer control horizons:fig:problematic points}(left) a typical trajectory tends towards the origin quickly. However, Figure \ref{Reducing the optimization horizon using longer control horizons:fig:problematic points}(right) shows that for the set of initial values $\mathcal{X} := \{ \left( \cos(2 \pi k / k_{\max}), \sin(2 \pi k / k_{\max}) \right) \mid k = 1, 2, \ldots, k_{\max} \}$ with $k_{\max} = 2^7$ there exists a nonempty subset of initial values for which we obtain $\alpha < 0$ in \eqref{Setup:thm:aposteriori:eq1}. Hence, stability cannot be deduced from Theorem \ref{Setup:thm:aposteriori}. Here, we like to mention that stability can be guaranteed for all initial values $x_0 \in \mathcal{X}$ by means of Theorem \ref{Setup:thm:aposteriori} if we choose $N = 4$. Note that in \cite[Section 6]{NP1997} it has been shown in a similar manner that the closed loop is stable if $N \geq 5$. We like to point out that the approach presented in \cite{NP1997} for the unconstrained linear quadratic case exploits the connection between stability properties of the system reflected by the solution of the corresponding Riccati equation and the monotonicity of the cost function, see also \cite{BGW1990}. 
\end{example}
\begin{figure}[!ht]
	\begin{center}
		\includegraphics[width=5cm]{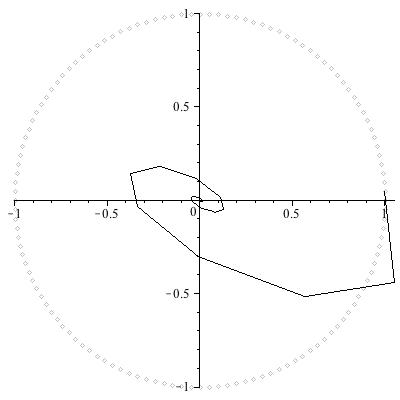}
		\hspace*{1cm}
		\includegraphics[width=5cm]{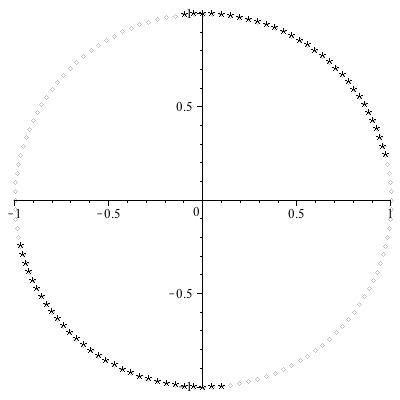}
		\caption{Left: Illustration of the used set of initial values (grey) together with a typical closed loop solution (black). Right: Illustration of the subset of initial values (black) of $\mathcal{X}$ (grey) for which $\alpha \geq 0$ in \eqref{Setup:thm:aposteriori:eq1} does not hold.}
		\label{Reducing the optimization horizon using longer control horizons:fig:problematic points}
	\end{center}
\end{figure}

Motivated by the fact that the computational effort grows rapidly with respect to the prediction horizon $N$, it is not desireable to choose $N$ larger than necessary to guarantee the relaxed Lyapunov inequality \eqref{Setup:thm:aposteriori:eq1} to hold. Our goal in this work is to develop an algorithm which allows us to check whether stability and performance in the sense of Theorem \ref{Setup:thm:aposteriori} can be guaranteed for a prediction horizon which is smaller than the prediction horizon needed in order to ensure \eqref{Setup:thm:aposteriori:eq1}. In order to further motivate this approach, Example \ref{Reducing the optimization horizon using longer control horizons:ex1} is considered again.

\begin{example}\label{Setup:ex:motivation m step}
	Instead of repeatedly considering only the first MPC step to evaluate $\alpha$ via \eqref{Setup:thm:aposteriori:eq1}, we define the respective quantity for every second MPC step via
	\begin{align*}
		\alpha_{N,2}(x_0) := \frac {V_N(x_0) - V_N(x_{\mu_N}(2;x_0))}{\sum_{n=0}^{1} \ell(x_{\mu_N}(n;x_0),\mu_N(x_{\mu_N}(n;x_0)))}
	\end{align*}
	and obtain $\alpha_{N,2}(x_0) > 0$ for each $x_0 \in \mathcal{X}$. Hence, defining $\alpha := \inf_{x_0 \in \mathcal{X}} \alpha_{N,2}(x_0) > 0$ yields
	\begin{align}
		\label{Setup:ex:motivation m step:eq1}
		V_N(x_0) \geq V_N(x_{\mu_N}(2;x_0)) + \alpha \left( \sum_{n=0}^{1} \ell(x_{\mu_N}(n;x_0),\mu_N(x_{\mu_N}(n;x_0))) \right),
	\end{align}
	i.e., the relaxed Lyapunov inequality with a positive suboptimality index $\alpha$ after two steps. Indeed, this conclusion holds true for all $x \in \bX = X$ which can be shown by using the MPC feedback computed in Example \ref{Reducing the optimization horizon using longer control horizons:ex1}. Hence, despite our conditions to be trajectory based, for this example asymptotic stability in the sense of Definition \ref{Setup:def:stability} can be concluded.
\end{example}

In particular, Example \ref{Setup:ex:motivation m step} shows that a generalized type of relaxed Lyapunov inequality similar to \eqref{Setup:ex:motivation m step:eq1} may hold after implementing several controls despite the fact that the central assumption \eqref{Setup:thm:aposteriori:eq1} of Theorem \ref{Setup:thm:aposteriori} is violated. Checking the relaxed Lyapunov inequality \eqref{Setup:thm:aposteriori:eq1} less often may help in order to ensure desired stability properties of the resulting closed loop. 
We aim at designing a strategy which ensures -- a priori and at runtime of the corresponding MPC algorithm -- that a relaxed Lyapunov inequality is fulfilled after $m \in \{1,2,\ldots,N-1\}$ steps.

Our first attempt is motivated by an observation from \cite[Section 7]{GPSW2010}. In this reference estimates for the suboptimality degree $\alpha$ for a set of initial conditions are deduced and the following fact has been proven: if more than one element of the computed sequence of control values is applied, then the suboptimality estimate $\alpha$ is increasing (up to a certain point). To incorporate this idea into our MPC scheme, we first need to extend our notation. The list $\cS = (\sigma(0), \sigma(1), \ldots) \subseteq \N_{0}$ is introduced, which is assumed to be in ascending order, in order to indicate time instants at which the control sequence is updated. The closed loop solution at time instant $\sigma(n)$ is denoted by $x_{n} = x_{\mu_{N}}(\sigma(n))$. Furthermore, the abbreviation $m_{n} := \sigma(n + 1) - \sigma(n)$, i.e., the time between two MPC updates, is used. Hence, 
\begin{align}
	\label{Reducing the optimization horizon using longer control horizons:eq:closed loop}
	x_{\mu_{N}}(\sigma(n) + m_{n}) = x_{\mu_{N}}(\sigma(n + 1)) = x_{n + 1}
\end{align}
holds. This enables us -- in view of Bellman's principle of optimality -- to define the closed loop control
\begin{align}
	\label{Reducing the optimization horizon using longer control horizons:eq:closed loop control}
	\mu_{N}^{\cS}(\cdot; x_{n}) := \argmin_{\{u \in \bU^{m_{n}}:x_u(k;x_n) \in \bX\ \text{ for $k = 1,2,\ldots,m_n$}\}} \Big\{ V_{N - m_{n}}(x_{u}(m_{n}; x_{n})) + \sum\limits_{k = 0}^{m_{n} - 1} \l(x_{u}(k; x_{n}), u(k)) \Big\}.
\end{align}

Describing the fact shown in \cite[Section 7]{GPSW2010} more precisely, a lower bound on the degree of suboptimality $\alpha_{N, m_{n}}$ relative to the horizon length $N$ and the number of controls to be implemented $m_{n}$ can be obtained. This bound allows for measuring the tradeoff between the infinite horizon cost induced by the MPC feedback law $\mu_{N}^{\cS}(\cdot; \cdot)$ similar to Theorem \ref{Setup:thm:aposteriori}, i.e.
\begin{align}
	\label{Reducing the optimization horizon using longer control horizons:eq:value function mu_{N}}
	V_{\infty}^{\mu_{N}^{\cS}} (x_{0}) := \sum\limits_{n = 0}^\infty\sum\limits_{k = 0}^{m_{n} - 1} \l\left( x_{\mu_{N}^{\cS}}(k; x_{n}), \mu_{N}^{\cS}(k; x_{n}) \right),
\end{align}
and the infinite horizon optimal value function $V_{\infty}(x_0)$ evaluated at $x_0$. 
We point out that the results shown in \cite{GPSW2010} ensure stability for a set of initial values. Hence, this approach may lead to a conservative performance estimate at least with respect to parts of the state space. Here, we extend \eqref{Setup:thm:aposteriori:eq1} to an $m$-step relaxed Lyapunov inequality which is similar to \cite{GPSW2010} but applied in a trajectory based setting. Note that the controllability condition, which was used in order to derive these results, is, in general, difficult to check if state constraints have to be taken into account.
\begin{proposition}\label{Reducing the optimization horizon using longer control horizons:prop:trajectory estimate}
	Let a list $\cS = (\sigma(n))_{n \in \N_0} \subseteq \N_0$ such that $m_n=\sigma(n+1)-\sigma(n) \in \{1,2,\ldots,N-1\}$ holds be given. Consider the open loop system $x_{u_{N}}(\cdot; \cdot)$ of \eqref{Setup:open loop system}, the feedback law $\mu_{N}^{\cS}(\cdot; \cdot)$, the closed loop trajectory $x_{n}$, $n \in \N_0$, of \eqref{Reducing the optimization horizon using longer control horizons:eq:closed loop} with initial value $x_0 \in \bX$ and a fixed $\overline{\alpha} \in (0, 1]$. If there exists a function $V_{N}: \bX \rightarrow \R_{0}^{+}$ satisfying
	\begin{align}
		\label{Reducing the optimization horizon using longer control horizons:prop:trajectory estimate:eq1}
		V_{N}(x_{n}) \geq V_{N}(x_{n + 1}) + \overline{\alpha} \sum\limits_{k = 0}^{m_{n} - 1} \l(x_{u_{N}}(k; x_{n}), u_{N}(k; x_{n}))
	\end{align}
	with $m_{n} \in \{1, \ldots, N - 1\}$ for all $n \in \N_{0}$, then
	\begin{align}
		\label{Reducing the optimization horizon using longer control horizons:prop:trajectory estimate:eq2}
		\overline{\alpha} V_{\infty}(x_{0}) \leq \overline{\alpha} V_{\infty}^{\mu_{N}^{\cS}}(x_{0}) \leq V_{N}(x_{0}) \leq V_{\infty}(x_{0})
	\end{align}
	holds for $\mu_{N}^{\cS}(\cdot; \cdot)$ given by \eqref{Reducing the optimization horizon using longer control horizons:eq:closed loop control} for all $n \in \N_{0}$.
\end{proposition}
\begin{proof}
	Reordering \eqref{Reducing the optimization horizon using longer control horizons:prop:trajectory estimate:eq1}, we obtain $\overline{\alpha} \sum_{k = 0}^{m_{n} - 1} \l(x_{u_{N}}(k; x_{n}), u_{N}(k; x_{n})) \leq V_{N}(x_{n}) - V_{N}(x_{n + 1})$. Summing over all the first $n \in \N_{0}$ time instants yields
	\begin{align*}
		\overline{\alpha} \sum_{i = 0}^{n} \sum_{k = 0}^{m_i - 1} \l(x_{u_{N}}(k; x_i), u_{N}(k; x_i)) \leq V_{N}(x_{0}) - V_{N}(x_{n + 1}) \leq V_{N}(x_{0}).
	\end{align*}
	Hence, by definition of $\mu_{N}^{\cS}(\cdot; \cdot)$ in \eqref{Reducing the optimization horizon using longer control horizons:eq:closed loop control} and $V_{\infty}^{\mu_{N}^{\cS}} (\cdot)$ in \eqref{Reducing the optimization horizon using longer control horizons:eq:value function mu_{N}}, taking $n$ to infinity implies the second inequality in \eqref{Reducing the optimization horizon using longer control horizons:prop:trajectory estimate:eq2}. The first and the last inequality in \eqref{Reducing the optimization horizon using longer control horizons:prop:trajectory estimate:eq2} follow by the definition of the value functions $V_N$ and $V_\infty$ 
	which concludes the proof.
\end{proof}
	
An implementation which aims at guaranteeing a fixed lower bound of the degree of suboptimality $\overline{\alpha}$ in the sense of Proposition \ref{Reducing the optimization horizon using longer control horizons:prop:trajectory estimate} may take the following form:
\begin{fshaded}
	\begin{algorithm}\label{Reducing the optimization horizon using longer control horizons:alg:basic algorithm}
	Given state $x_{0} := x$, $n = 0$, list $\cS = (n)$, $N \in \mathbb{N}_{\geq 2}$ and $\overline{\alpha} \in (0, 1]$
	\begin{itemize}
		\item[(I)] Set $j := 0$ and compute $u_{N}(\cdot; x_{n})$ and $V_{N}(x)$. Do
		\begin{itemize}
			\item[(a)] Set $j := j + 1$, compute $V_{N}(x_{u_{N}}(j; x_{n}))$
			\item[(b)] Compute $\alpha = \{ \overline{\alpha} \mid \overline{\alpha} \text{ maximally satisfies \eqref{Reducing the optimization horizon using longer control horizons:prop:trajectory estimate:eq1} with $m_n := j$} \}$
			\item[(c)] If $\alpha \geq \overline{\alpha}$: Set $m_{n} := j$ and goto (II)
			\item[(d)] If $j = N - 1$: Set $m_{n}$ according to exit strategy and goto (II)
		\end{itemize}
		while $\alpha < \overline{\alpha}$
		\item[(II)] For $j = 1, \ldots, m_{n}$ do
		\begin{itemize}
			\item[] Implement $\mu_{N}^{\cS}(j - 1; x_{n}) := u_{N}(j-1; x_{n})$
		\end{itemize}
		\item[(III)] Set $\cS := (\cS, \mbox{back}(\cS) + m_{n})$, $x_{n + 1} := x_{\mu_{N}^{\cS}}(m_{n}; x_{n})$, $n := n + 1$ and goto (I)
	\end{itemize}
	\end{algorithm}
	\vspace*{-0.1cm}
\end{fshaded}

Here, we adopted the programming notation \textit{back} which allows for fast access to the last element of a list. Note that $\cS$ is built up during runtime of the algorithm and not known in advance. Hence, $\cS$ is always ordered.
\begin{remark}\label{Reducing the optimization horizon using longer control horizons:rem:exit strategy}
	If \eqref{Reducing the optimization horizon using longer control horizons:prop:trajectory estimate:eq1} is not satisfied for $j \leq N - 1$, an exit strategy has to be used since the performance bound $\overline{\alpha}$ cannot be guaranteed. In order to cope with this issue, there exist remedies, e.g., one may increase the prediction horizon and repeat Step (I). While local validity of \eqref{Reducing the optimization horizon using longer control horizons:prop:trajectory estimate:eq1} can be ensured for sufficiently large $N$, extending the prediction horizon typically results in prolonged computing times. Therefore, if realtime guarantees are required, a respective bound for the initial length of the horizon needs to be satisfied. 
	Additionally, the proof of Proposition \ref{Reducing the optimization horizon using longer control horizons:prop:trajectory estimate} cannot be applied in this context due to the prolongation of the horizon. Yet, it can be replaced by estimates from \cite{GP2011,G2010} for varying prediction horizons to obtain a result similar to \eqref{Reducing the optimization horizon using longer control horizons:prop:trajectory estimate:eq2}. Alternatively, one may continue with the algorithm. If the exit strategy does not have to be called again, the algorithm guarantees the desired performance for $x_{n}$ instead of $x_{0}$, i.e., from that point on.
\end{remark}
Utilizing Algorithm \ref{Reducing the optimization horizon using longer control horizons:alg:basic algorithm} the following result shows asymptotically stable behavior of the computed state trajectory:
\begin{theorem}\label{Reducing the optimization horizon using longer control horizons:thm:stability}
	Suppose a control system \eqref{Setup:nonlinear discrete time system} with initial value $x_{0} \in \bX$ and $\overline{\alpha} \in (0, 1]$ to be given and apply Algorithm \ref{Reducing the optimization horizon using longer control horizons:alg:basic algorithm}. Assume that for each iterate $n \in \N_{0}$ condition $\alpha \geq \overline{\alpha}$ in Step (Ic) of Algorithm \ref{Reducing the optimization horizon using longer control horizons:alg:basic algorithm} is satisfied for some $j \in \{ 1, \ldots, N - 1 \}$. Then, the closed loop trajectory corresponding to the closed loop control $\mu_{N}^{\cS}(\cdot; \cdot)$ resulting form Algorithm \ref{Reducing the optimization horizon using longer control horizons:alg:basic algorithm} satisfies the performance estimate \eqref{Reducing the optimization horizon using longer control horizons:prop:trajectory estimate:eq2}. If, in addition, the conditions of Theorem \ref{Setup:thm:aposteriori}(ii) hold, then $x_{\mu_{N}^{\cS}}(\cdot)$ behaves like a trajectory of an asymptotically stable system.
\end{theorem}
\begin{proof}
	The algorithm constructs the set $\cS$. Since Step (Ic) is satisfied for some $j \in \{ 1, \ldots, N - 1 \}$ the assumptions of Proposition \ref{Reducing the optimization horizon using longer control horizons:prop:trajectory estimate}, i.e., Inequality \eqref{Reducing the optimization horizon using longer control horizons:prop:trajectory estimate:eq1} for $x_{n} = x_{\mu_{N}^{\cS}}(\sigma(n))$, are satisfied. Hence, by Proposition \ref{Reducing the optimization horizon using longer control horizons:prop:trajectory estimate}, the performance Estimate \eqref{Reducing the optimization horizon using longer control horizons:prop:trajectory estimate:eq2} follows for the control sequence $\mu_{N}^{\cS}(\cdot; \cdot)$. Last, similar to the proof of Theorem \ref{Setup:thm:aposteriori} given in \cite[Theorem 7.6]{GP2011}, standard direct Lyapunov techniques can be applied to conclude asymptotically stable behavior of the closed loop trajectory $x_{\mu_{N}^{\cS}}(\cdot)$.
\end{proof}
\begin{example}\label{Reducing the optimization horizon using longer control horizons:ex2}
	Consider Example \ref{Reducing the optimization horizon using longer control horizons:ex1} in the context of Theorem \ref{Reducing the optimization horizon using longer control horizons:thm:stability}. Recall that using the Riccati based open loop control law $u_N(\cdot; x(n))$ in the standard MPC fashion $\mu_N(x(n)) := u_N(0; x(n))$ for horizon length $N = 3$ together with Theorem \ref{Setup:thm:aposteriori}, there exist initial values for which $\alpha > 0$ cannot be guaranteed. Now, if we use Algorithm \ref{Reducing the optimization horizon using longer control horizons:alg:basic algorithm} instead, i.e. we allow for $m_n > 1$, then we obtain $\alpha > \overline{\alpha} = 0$ for each initial value $x_0 \in X$. In Figure \ref{Reducing the optimization horizon using longer control horizons:fig:decrease in V} we illustrated the impact of changing $m_n$ on the difference $V_N(x_n) - V_N(x_{n+1})$ for $N = 3$. In particular, Figure \ref{Reducing the optimization horizon using longer control horizons:fig:decrease in V}(left) shows that there exists regions in the state space where the value function is increasing whereas $V_3(\cdot)$ in Figure \ref{Reducing the optimization horizon using longer control horizons:fig:decrease in V}(right) is always decreasing, i.e. all trajectories converge towards the origin.
	\begin{figure}[!ht]
		\begin{center}
			\includegraphics[width=5cm]{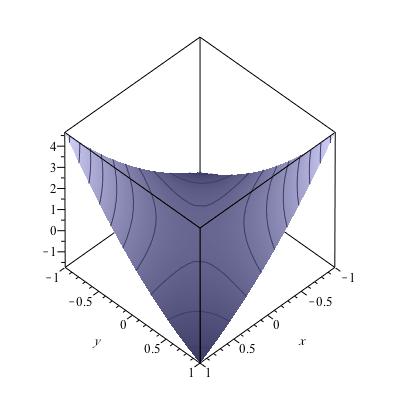}
			\hspace*{0.5cm}
			\includegraphics[width=5cm]{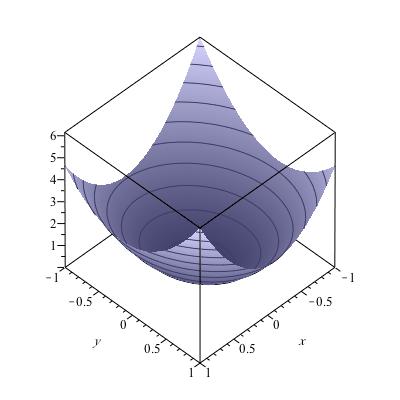}
		\end{center}
		\caption{Left: Difference $V_3(x_n) - V_3(x_{n+1})$ with $m_n = 1$. Right: Difference $V_3(x_n) - V_3(x_{n+1})$ for with $m_n = 2$.}
		\label{Reducing the optimization horizon using longer control horizons:fig:decrease in V}
	\end{figure}

	Using $\alpha_1(\| x \|) = \alpha_3(\| x \|) = \| x \|^2 = \min_{u \in U} \ell(x,u) \leq V_N(x) = x^T P_N x \leq \overline{\lambda}(P) \| x \|^2 = \alpha_2(\| x \|)$ where $\overline{\lambda}$ denotes the largest Eigenvalue of the matrix $P_N$ ensures the assumptions of Theorem \ref{Setup:thm:aposteriori}(ii). Hence, applying Theorem \ref{Reducing the optimization horizon using longer control horizons:thm:stability} enables us to conclude asymptotic stability of the closed loop generated by Algorithm \ref{Reducing the optimization horizon using longer control horizons:alg:basic algorithm}. 
\end{example}

Example \ref{Reducing the optimization horizon using longer control horizons:ex2} indicates that checking the relaxed Lyapunov inequality less often may allow us to maintain stability for a reduced prediction horizon. We point out that in Example \ref{Reducing the optimization horizon using longer control horizons:ex2} condition \eqref{Reducing the optimization horizon using longer control horizons:prop:trajectory estimate:eq1} has been ensured in advance, i.e., before implementing a control input at the plant. Algorithm \ref{Reducing the optimization horizon using longer control horizons:alg:basic algorithm} may vary the number of open loop control values to be implemented during runtime at each MPC iteration. In particular, the system may stay in open loop for more than one sampling period in order to guarantee \eqref{Reducing the optimization horizon using longer control horizons:prop:trajectory estimate:eq1} to hold. Such a procedure may lead to severe problems in terms of robustness, cf. \cite{B1995, MS2007}. Hence, from a practical point of view it is preferable to close the control loop as often as possible, i.e., $m_{n} = 1$ for all $n \in \N_{0}$. Furthermore, for many applications stable behavior of the closed loop is observed for $m_{n} = 1$ even if stability cannot be guaranteed via \eqref{Setup:thm:aposteriori:eq1}, cf. Example \ref{Setup:ex:motivation m step}. In the following, we deduce conditions which allow the control loop to be closed more often compared to Algorithm \ref{Reducing the optimization horizon using longer control horizons:alg:basic algorithm} while maintaining stability.

\section{Robustification by Closing the Loop more often}
\label{Section:Robustification by reducing the control horizon}

If $m_{n} > 1$ is required in Step (I) of the proposed algorithm in order to ensure \eqref{Reducing the optimization horizon using longer control horizons:prop:trajectory estimate:eq1}, the following methodology which has been proposed in \cite[Section 4]{PW2010} allows us to prove the following: Given a certain condition, the degree of suboptimality $\overline{\alpha}$ in \eqref{Reducing the optimization horizon using longer control horizons:prop:trajectory estimate:eq1} is maintained for the updated control
\begin{align}
	\label{Robustification by reducing the control horizon:eq:control update}
	\hat{u}_{N}(k; x_{n}) := \begin{cases}
				u_{N}(k; x_{n}), & k \leq j - 1 \\
				u_{N}(k - j; x_{u_{N}}(j; x_{n})), & k \geq j
				\end{cases}
\end{align}
with $j \in \{ 1, \ldots, m_{n} - 1 \}$ and where $\sigma(n) + j$ is added to the list $\cS$. The theoretical foundation of such a method is given by the following result:
\begin{proposition}\label{Robustification by reducing the control horizon:prop:m2one estimate}
	Let the open loop system $x_{u_{N}}(\cdot; \cdot)$ of \eqref{Setup:open loop system} with initial value $x_{n} \in \bX$ be given and inequality \eqref{Reducing the optimization horizon using longer control horizons:prop:trajectory estimate:eq1} hold for $u_{N}(\cdot;x_{n})$, $\overline{\alpha} \in (0, 1]$, and $m_{n} \in \{2, \ldots, N - 1\}$. If the inequality
	\begin{align}
		\label{Robustification by reducing the control horizon:prop:m2one estimate:eq1}
		& V_{N}(x_{u_{N}}(m_{n} - j; x_{u_{N}}(j; x_{n}))) - V_{N - j}(x_{u_{N}}(j; x_{n})) \\
		& \leq  ( 1 - \overline{\alpha}) \sum\limits_{k = 0}^{j - 1} \l(x_{u_{N}}(k; x_{n}), u_{N}(k; x_{n})) - \overline{\alpha} \sum\limits_{k = j}^{m_{n} - 1} \l(x_{u_{N}}(k - j; x_{u_{N}}(j; x_{n})), u_{N}(k - j; x_{u_{N}}(j; x_{n}))) \nonumber
	\end{align}
	holds for some $j \in \{ 1, \ldots, m_{n} - 1 \}$, then the control sequence $u_{N}(\cdot; x_{n})$ can be replaced by \eqref{Robustification by reducing the control horizon:eq:control update} and the lower bound on the degree of suboptimality $\overline{\alpha}$ is locally maintained.
\end{proposition}
\begin{proof}
	In order to show the assertion, we need to show \eqref{Reducing the optimization horizon using longer control horizons:prop:trajectory estimate:eq1} for the modified control sequence \eqref{Robustification by reducing the control horizon:eq:control update}. Reformulating \eqref{Robustification by reducing the control horizon:prop:m2one estimate:eq1} by shifting the running costs associated with the unchanged control to the left hand side of \eqref{Robustification by reducing the control horizon:prop:m2one estimate:eq1} we obtain
	\begin{align*}
		& V_{N}(x_{u_{N}}(m_{n} - j; x_{u_{N}}(j; x_{n}))) - V_{N}(x_{u_{N}}(0; x_{n}))  \leq - \overline{\alpha} \sum\limits_{k = 0}^{m_{n} - 1} \l(x_{\hat{u}_{N}}(k; x_{n}), \hat{u}_{N}(k; x_{n}))
	\end{align*}
	which is equivalent to
	the relaxed Lyapunov inequality \eqref{Reducing the optimization horizon using longer control horizons:prop:trajectory estimate:eq1} for the updated control $\hat{u}_{N}(\cdot; x_{n})$.
\end{proof}
Utilizing Proposition \ref{Robustification by reducing the control horizon:prop:m2one estimate} in Algorithm \ref{Reducing the optimization horizon using longer control horizons:alg:basic algorithm} we see that only Steps (II) and (III) need to be changed and may take the following form:
\begin{fshaded}
	\begin{algorithm}[Modification of Algorithm \ref{Reducing the optimization horizon using longer control horizons:alg:basic algorithm}]\label{Robustification by reducing the control horizon:alg:first improvement}$\\ $ \vspace*{-0.5cm}
	\begin{itemize}
		\item[(II)] Set $\hat{n} := 1$ and $s := \mbox{back}(\cS)$. For $j = 1, \ldots, m_{n}$ do
		\begin{itemize}
			\item[(a)] Implement $\mu_{N}^{\cS}(j-1; x_{n}) := u_{N}(j-1; x_{n})$
			\item[(b)] Compute $u_{N}(\cdot; x_{u_{N}}(j; x_{n}))$, construct $\hat{u}_{N}(\cdot; x_{n})$ according to \eqref{Robustification by reducing the control horizon:eq:control update} and compute $V_{N}(x_{\hat{u}_{N}}(m_{n}; x_{n}))$
			\item[(c)] If condition \eqref{Robustification by reducing the control horizon:prop:m2one estimate:eq1} holds:
			\begin{itemize}
				\item[] Update $u_{N}(\cdot; x_{n}) := \hat{u}_{N}(\cdot; x_{n})$. If $j < m_{n}$ set $\cS := ( \cS, s + j )$, $x_{n + \hat{n}} := x_{\mu_{N}^{\cS}}(j; x_{n})$ and $\hat{n} = \hat{n} + 1$. 
			\end{itemize}
		\end{itemize}
		\item[(III)] Set $\cS := ( \cS, s + m_n )$, $x_{n + \hat{n}} := x_{\mu_{N}^{\cS}}(m_{n}; x_{n})$, $n := n + \hat{n}$ and goto (I)
	\end{itemize}
	\end{algorithm}
	\vspace*{-0.1cm}
\end{fshaded}
Due to the principle of optimality, the value of $V_{N - j}(x_{u_{N}}(j; x_{n}))$ in \eqref{Robustification by reducing the control horizon:prop:m2one estimate:eq1} is known in advance from $V_{N}(x_{n})$. Hence, only $u_{N}(\cdot; x_{u_{N}}(j; x_{n}))$ and $V_{N}(x_{\hat{u}_n}(m_{n}; x_{n}))$ have to be computed. This result has to be checked with $V_{N}(x_{n})$ for all $j \in \{1, \ldots, m_{n} - 2\}$. Hence, the updating instant $\sigma(n)$ has to be kept in mind. 
We also like to stress the fact that condition \eqref{Robustification by reducing the control horizon:prop:m2one estimate:eq1} allows for a less fast decrease of energy along the closed loop, i.e., the case $V_{N}(x_{\hat{u}_{N}}(m_{n}; x_{n})) \geq V_{N}(x_{u_{N}}(m_{n}; x_{n}))$ is not excluded in general which is illustrated by the following example.
\begin{example}
	Consider Example \ref{Reducing the optimization horizon using longer control horizons:ex1} and the initial values $\tilde{x}_0 = (0,1)^T$ with prediction horizon $N=3$. If we apply Algorithm \ref{Robustification by reducing the control horizon:alg:first improvement} we obtain $V_N(\tilde{x}_0) = 5.109994744$ and $V_N(x_{\mu_N}(2;\tilde{x})) = 2.83461176$ which yields $\alpha = 0.5136$. If Algorithm \ref{Reducing the optimization horizon using longer control horizons:alg:basic algorithm} is employed we obtain $V_N(x_{u_N}(2;\tilde{x})) = 2.827656536$ which implies $\alpha = 0.5144$. Hence, Algorithm \ref{Robustification by reducing the control horizon:alg:first improvement} accepts also deteriorations as long as the desired suboptimality degree is still maintained. \\
	Taking $\overline{x}_0 = (1,0)^T$ as initial value shows that the impact of Algorithm \ref{Robustification by reducing the control horizon:alg:first improvement} may also improve these key figures: In this case Algorithm \ref{Robustification by reducing the control horizon:alg:first improvement} provides $V_N(\overline{x}) = 4.08117251$, $V_N(x_{\mu_N}(2;\overline{x})) = 0.96290399$ with $\alpha = 0.7733$ whereas Algorithm \ref{Reducing the optimization horizon using longer control horizons:alg:basic algorithm} gives us $V_N(x_{u_N}(2;\overline{x})) = 1.22718283$ with $\alpha = 0.7470$. For a further comparison we refer to the numerical experiments in the ensuing section.
\end{example}

Next, a counterpart to Theorem \ref{Reducing the optimization horizon using longer control horizons:thm:stability} based on Algorithm \ref{Robustification by reducing the control horizon:alg:first improvement} instead of Algortihm \ref{Reducing the optimization horizon using longer control horizons:alg:basic algorithm} is established. This results allows us to verify that modifying and, thus, robustifying the algorithm still leads to the desired stability and performance properties. To this end, Proposition \ref{Robustification by reducing the control horizon:prop:m2one estimate} is applied iteratively to show asymptotically stable behavior of the state trajectory generated by Algorithm \ref{Reducing the optimization horizon using longer control horizons:alg:basic algorithm}:

\begin{theorem}\label{Robustification by reducing the control horizon:thm:stability}
	Let a control system \eqref{Setup:nonlinear discrete time system} with initial value $x_{0} \in \bX$ and $\overline{\alpha} \in (0, 1]$ be given. Furthermore, suppose Algorithm \ref{Reducing the optimization horizon using longer control horizons:alg:basic algorithm} with the modification of Algorithm \ref{Robustification by reducing the control horizon:alg:first improvement} is applied. Assume that the condition $\alpha \geq \overline{\alpha}$ in Step (Ic) of Algorithm \ref{Reducing the optimization horizon using longer control horizons:alg:basic algorithm} is satisfied for some $j \in \{1,\ldots,N-1\}$ for each iterate $n \in \N_0$. Then the closed loop trajectory corresponding to the closed loop control $\mu_{N}^{\cS}(\cdot; \cdot)$ resulting from the used algorithm satisfies the performance estimate \eqref{Reducing the optimization horizon using longer control horizons:prop:trajectory estimate:eq2} from Proposition \ref{Reducing the optimization horizon using longer control horizons:prop:trajectory estimate}. If, in addition, the conditions of Theorem \ref{Setup:thm:aposteriori}(ii) hold, then $x_{\mu_{N}^{\cS}}(\cdot)$ behaves like a trajectory of an asymptotically stable system.
\end{theorem}
\begin{proof}
	The list $\cS$ constructed by the algorithm contains all time instants at which the sequence of control values is updated by the MPC feedback law $\mu_N^{\cS}(\cdot;\cdot)$. Hence, different to the basic Algorithm \ref{Reducing the optimization horizon using longer control horizons:alg:basic algorithm} the list $\cS$ may not only contain time instants at which the relaxed Lyapunov inequality \eqref{Reducing the optimization horizon using longer control horizons:prop:trajectory estimate:eq1} holds.\\
	Consider a second list $\widetilde{\cS} = (\tau(n))_{n \in \N_0}$ constructed analogously to $\cS$ from Algorithms \ref{Reducing the optimization horizon using longer control horizons:alg:basic algorithm}, \ref{Robustification by reducing the control horizon:alg:first improvement} but which is not updated in the modified Step (II) of Algorithm \ref{Robustification by reducing the control horizon:alg:first improvement}. Since Step (Ic) is satisfied for some $j \in \{1, \ldots, N - 1\}$, Inequality \eqref{Reducing the optimization horizon using longer control horizons:prop:trajectory estimate:eq1} with $x_{n} = x_{\mu_{N}^{\cS}}(\tau(n))$ and $m_n = \tau(n+1) - \tau(n)$ is satisfied. Using Proposition \ref{Robustification by reducing the control horizon:prop:m2one estimate} ensures that this inequality is maintained despite of updates carried out by Algorithm \ref{Robustification by reducing the control horizon:alg:first improvement}. Hence, by Proposition \ref{Reducing the optimization horizon using longer control horizons:prop:trajectory estimate}, the performance estimate \eqref{Reducing the optimization horizon using longer control horizons:prop:trajectory estimate:eq2} follows for the control sequence $\mu_{N}^{\cS}(\cdot; \cdot)$. Again, similar to the proof of Theorem \ref{Setup:thm:aposteriori} given in \cite[Theorem 7.6]{GP2011}, standard direct Lyapunov techniques can be used to show asymptotically stable behavior of the closed loop $x_{\mu_{N}^{\cS}}(\cdot)$.
\end{proof}

An important conclusion of Proposition \ref{Robustification by reducing the control horizon:prop:m2one estimate} and Theorem \ref{Robustification by reducing the control horizon:thm:stability} is the following:

\begin{corollary}\label{Robustification by reducing the control horizon:cor:stability}
	Consider the open loop system $x_{u_{N}}(\cdot; \cdot)$ of \eqref{Setup:open loop system}, the feedback law $\mu_{N}^{\cS}(\cdot; \cdot)$, the closed loop trajectory $x_{n}$, $n \in \N_0$, of \eqref{Reducing the optimization horizon using longer control horizons:eq:closed loop} with initial value $x_0 \in \bX$ and a fixed $\overline{\alpha} \in (0, 1]$ to be given. Moreover, suppose inequality \eqref{Reducing the optimization horizon using longer control horizons:prop:trajectory estimate:eq1} to hold for $u_{N}(\cdot;x_{n})$, $\overline{\alpha} \in (0, 1]$, and $m_{n} \in \{2, \ldots, N - 1\}$. If \eqref{Robustification by reducing the control horizon:prop:m2one estimate:eq1} holds for all $j \in \{ 1, \ldots, m_n - 1 \}$ and all $n \in \N_{0}$, then \eqref{Setup:thm:aposteriori:eq2} holds, that is the standard MPC feedback $\mu_{N}(\cdot)$ can be applied. If, in addition, the conditions of Theorem \ref{Setup:thm:aposteriori}(ii) hold, then $x_{\mu_{N}}(\cdot)$ behaves like a trajectory of an asymptotically stable system.
\end{corollary}
\begin{proof}
	Follows directly from Theorem \ref{Robustification by reducing the control horizon:thm:stability}.
\end{proof}
The following example illustrates the impact of the modified algorithm, but it also indicates that the lower bound may still not be tight.
\begin{example}\label{Robustification by reducing the control horizon:ex}
	In Example \ref{Reducing the optimization horizon using longer control horizons:ex2} stability can be shown if one allows for implementing more than than one element of the open loop sequence of control values, i.e., $m_n \geq 1$. Now we use Algorithm \ref{Robustification by reducing the control horizon:alg:first improvement} instead of Algorithm \ref{Reducing the optimization horizon using longer control horizons:alg:basic algorithm}, i.e. we allow for longer control horizons $m_n$ and try to verify whether the control loop can be closed more often. And indeed, we obtain $\alpha > \overline{\alpha} = 0$ for each initial value $x_0 \in \mathcal{X}$ with $m_n = 1$, that is we can show stability for standard MPC according to Corollary \ref{Robustification by reducing the control horizon:cor:stability}. Moreover, taking Example \ref{Reducing the optimization horizon using longer control horizons:ex2} into account, this assertion holds for all $x \in X = \R^2$.

	Yet, if the suboptimality bound is slightly increased to $\overline{\alpha} = 0.01$, the condition $\alpha \geq \overline{\alpha}$ in Step (Ic) is not satisfied for any $j \in \{ 1, \ldots, N - 1 \}$ for trajectories emanating from a subset of $\mathcal{X}$, cf. Figure \ref{Robustification by reducing the control horizon:fig:critical points}.
\end{example}
\begin{figure}[!ht]
	\begin{center}
		\includegraphics[width=5cm]{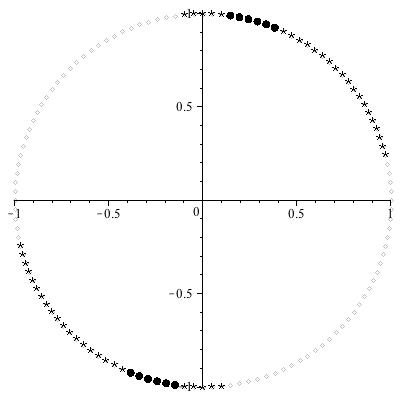}
		\caption{Illustration of initial values (grey) where $\alpha < 0$ in \eqref{Setup:thm:aposteriori:eq1} ($\star$) whereas $\alpha < \overline{\alpha} = 0.01$ in Step (Ic) of Algorithm \ref{Reducing the optimization horizon using longer control horizons:alg:basic algorithm} together with Algorithm \ref{Robustification by reducing the control horizon:alg:first improvement} for at least one iterate $n$ ($\bullet$)}
		\label{Robustification by reducing the control horizon:fig:critical points}
	\end{center}
\end{figure}
Intuitively, one could guess that the performance of the Riccati based feedback law in our example is better than $\overline{\alpha} = 0.01$. In the following section, we address this issue indirectly as an outcome of an exit strategy required in Step (Id) of Algorithm \ref{Reducing the optimization horizon using longer control horizons:alg:basic algorithm}, see also Remark \ref{Reducing the optimization horizon using longer control horizons:rem:exit strategy}.

\section{Acceptable Violations of Relaxed Lyapunov Inequality}
\label{Section:Acceptable violations of relaxed Lyapunov inequality}

Until now we have supposed that our relaxed Lyapunov Inequality \eqref{Reducing the optimization horizon using longer control horizons:prop:trajectory estimate:eq1}, i.e.,
\begin{align}\label{Acceptable violations of relaxed Lyapunov inequality:eq:rho}
	\rho_i = \rho_i(\overline{\alpha}) := V_{N}(x_{i}) - V_{N}(x_{u_{N}}(m_{i}; x_{i})) - \overline{\alpha} \sum\limits_{k = 0}^{m_{i} - 1} \l(x_{u_{N}}(k; x_{i}), u_{N}(k; x_{i})) \geq 0,
\end{align}
is satisfied for some $m_{i} \in \{1, 2, \ldots, N - 1 \}$. It is well known that \eqref{Reducing the optimization horizon using longer control horizons:prop:trajectory estimate:eq1} and, thus, $\rho_i \geq 0$ holds for sufficiently long horizon $N$, see \cite{JH2005, GMTT2005, AB1995, GPSW2010}. However, this is not necessarily true for short horizons $N$ -- even if the closed loop shows asymptotically stable behavior. Our basic Algorithm \ref{Reducing the optimization horizon using longer control horizons:alg:basic algorithm} allows us to cope with such a case via an exit strategy in Step (Id). As outlined in Remark \ref{Reducing the optimization horizon using longer control horizons:rem:exit strategy}, the length of the prediction horizon could be sufficiently increased in order to deal with this issue --- a remedy which should be avoided due to its high computational costs. Here, we pursue a different approach based on the following result which uses an idea similar to the watchdog technique in nonlinear optimization, see \cite{CPLP1982}.

\begin{theorem}\label{Acceptable violations of relaxed Lyapunov inequality:thm:estimate}
	Let a list $\cS = (\sigma(0),\sigma(1),\ldots) \subseteq \N_0$ be given such that the sequence $(m_i)_{i \in \N_0}$ defined by $m_i := \cS(i+1) - \cS(i)$ is contained in $\{1,2,\ldots,N-1\}$. In addition, for optimization horizon $N$ and initial state $x_0 \in \bX$, let the closed loop trajectory $(x_i)_{i \in \N_0}$ be generated by the feedback law $\mu_{N}^{\cS}(\cdot; \cdot)$ according to \eqref{Reducing the optimization horizon using longer control horizons:eq:closed loop}. Consider the open loop system $x_{u_N}(\cdot;\cdot)$ of \eqref{Setup:open loop system}, $\overline{\alpha} \in (0, 1]$, and the sequence $(\rho_i)_{i \in \N_0}$ from \eqref{Acceptable violations of relaxed Lyapunov inequality:eq:rho} to be given. Furthermore, suppose there exist $\alpha_1(\cdot)$, $\alpha_2(\cdot) \in \mathcal{K}_\infty$ such that $\alpha_2(\| x_i \|_{\xs}) \geq V_N(x_i)$ and $\l^\star(x_i) \geq \alpha_1(\| x_i \|_{\xs})$ hold for all $x_i$, $i \in \N_0$. Then, convergence of $(s_{n})_{n \in \N_0}$ with $s_n := \sum_{i = 0}^n \rho_i$ ensures the convergence $x_i \rightarrow \xs$ for $i$ tending to infinity, i.e., $x_{\mu_{N}^{\cS}}(\cdot)$ behaves like a trajectory of an asymptotically stable system. Furthermore, we have $V_N(x_i) \rightarrow 0$ for $i$ approaching infinity.
\end{theorem}
\begin{proof}
	Plugging the definition of $\rho_i$ into $s_n$ yields
	\begin{align*}
		s_n = \sum_{i=0}^n \rho_i & = \sum_{i=0}^n \left( V_N(x_i) - V_N(\underbrace{x_{u_N}(m_i;x_i)}_{=x_{i+1}}) - \overline{\alpha} \sum_{k=0}^{m_i-1} \ell(x_{u_N}(k;x_i),u_N(k;x_i)) \right) \\
		& = V_N(x_0) - V_N(x_{n+1}) - \sum_{i=0}^n \sum_{k=0}^{m_i-1} \ell(x_{u_N}(k;x_i),u_N(k;x_i)). 
	\end{align*}
	Since $V_N(x_{n+1}) \geq 0$ and $\ell(\cdot,\cdot)$ is positive the convergence of the sequence $(s_n)_{n \in \N_0}$ yields boundedness of each subtrahend. Using this assertion for the sum of the stage costs in combination with positivity of $\ell(\cdot,\cdot)$ and, thus, monotonicity of this summand with respect to the index $i$, implies that each summand of the stage costs $\sum_{k=0}^{m_i-1} \ell(x_{u_N}(k;x_i),u_N(k;x_i))$ tend to zero for $i$ approaching infinity. Then, the assertion with respect to the closed loop trajectory can be concluded by
	\begin{equation*}
		\| x_{\mu_{N}^{\cS}}(n;x_0) \|_{\xs} = \| x_{u_N}(k;x_i) \|_{\xs} \leq \alpha_1^{-1}(\ell^\star(x_{u_N}(k;x_i))) \leq \alpha_1^{-1}(\ell(x_{u_N}(k;x_i),\mu_{N}^{\cS}(k;x_{u_N}(k;x_i))))
	\end{equation*}
	for $n = \sigma(i) + k$. The latter also implies the assertion $V_N(x_i) \rightarrow 0$ in view of $\alpha_2(\|x_i\|_{\xs}) \geq V_N(x_i)$.
\end{proof}

Theorem \ref{Acceptable violations of relaxed Lyapunov inequality:thm:estimate} ensures asymptotic stability but does not guarantee the desired performance specification. We like to point out that $\rho_i$, $i \in \N_0$, does not have to be positive. Indeed, even $V_N(\cdot)$ may increase along the closed loop trajectory before it finally converges to zero. 

Analyzing the sequence $(s_n)_{n \in \N_0}$ and its limit more carefully leads to the following corollary which allows us to generalize Algorithm \ref{Reducing the optimization horizon using longer control horizons:alg:basic algorithm} by incorporating knowledge of the sequence $(s_n)_{n \in \N_0}$.

\begin{corollary}\label{Acceptable violations of relaxed Lyapunov inequality:cor:exit strategy}
	Let $\overline{\alpha} \in (0,1]$ be given and $s_n \geq 0$ hold for some $n \in \N_0$ with $s_n$ from Theorem \ref{Acceptable violations of relaxed Lyapunov inequality:thm:estimate}. Then, the closed loop trajectory $x_{\mu_N^\cS}(\cdot)$ satisfies the following relaxed Lyapunov inequality at $x_{n+1}$
	\begin{align}\label{Acceptable violations of relaxed Lyapunov inequality:cor:exit strategy:eq1}
		V_{N}(x_{n + 1}) + \overline{\alpha} \sum\limits_{i = 0}^{n} \sum\limits_{k = 0}^{m_i - 1} \l(x_{u_{N}}(k; x_i), u_{N}(k; x_i)) \leq V_{N}(x_{0}).
	\end{align}
	If $\lim_{n \to \infty} s_n \geq 0$ holds, then the suboptimality Estimate \eqref{Reducing the optimization horizon using longer control horizons:prop:trajectory estimate:eq2} is satisfied. 
\end{corollary}
\begin{proof}
	We obtain the stated relaxed Lyapunov inequality directly by inserting the definition of $\rho_i$ from \eqref{Acceptable violations of relaxed Lyapunov inequality:eq:rho} into $s_n$ and using the equivalence of the open and closed loop control $u_{N}(\cdot; x_{i})$, $\mu_{N}^{\cS}(\cdot; x_{i})$ from \eqref{Reducing the optimization horizon using longer control horizons:eq:closed loop control} which allows us to replace $x_{u_{N}}(m_{i}; x_{i})$ by $x_{i + 1}$. In order to show \eqref{Reducing the optimization horizon using longer control horizons:prop:trajectory estimate:eq2} we have to establish $\overline{\alpha} V_\infty^{\mu_{N}^{\cS}}(x_0) \leq V_N(x_0)$. To this end, we define
	\begin{align*}
		\tilde{s}_n := V_N(x_0) - \overline{\alpha} \sum_{i=0}^n \sum_{k=0}^{m_i-1} \l(x_{u_N}(k;x_i),u_N(k;x_i)) = s_n + V_N(x_{u_N}(m_i;x_i)).
	\end{align*}
	The fact that the range of $V_N(\cdot)$ and $\l(\cdot,u)$ is contained in $\R_0^+$ for all $u \in \U$ ensures that $(\tilde{s}_n)_{n \in \N_0}$ is a monotonically decreasing sequence which satisfies $\tilde{s}_n \geq s_n$ for all $n \in \N_0$. Since $\tilde{s}_n < 0$ for $n \in \N_0$ contradicts the positivity of $\lim_{n \to \infty} s_n$, we conclude $\tilde{s}_n \geq 0$ for all $n \in \N_0$. Hence, $(\tilde{s}_n)_{n \in \N_0}$ is monotonically decreasing and bounded and, thus, converges to $\tilde{s} \geq \lim_{n \to \infty} s_n \geq 0$. This yields
	\begin{align*}
		V_N(x_0) \geq \lim_{n \to \infty} \overline{\alpha} \sum_{i=0}^n \sum_{k=0}^{m_i-1} \l(x_{u_N}(k;x_i),u_N(k;x_i)) = \overline{\alpha} V_\infty^{\mu_{N}^{\cS}}(x_0),
	\end{align*}
	i.e., the desired assertion.
\end{proof}

The key idea of Corollary \ref{Acceptable violations of relaxed Lyapunov inequality:cor:exit strategy} in comparison to Theorem \ref{Setup:thm:aposteriori}(i) and Proposition \ref{Reducing the optimization horizon using longer control horizons:prop:trajectory estimate} is to allow for intermediate increases within \eqref{Setup:thm:aposteriori:eq1} or \eqref{Reducing the optimization horizon using longer control horizons:prop:trajectory estimate:eq1} for certain time instants $n$ which corresponds to $\rho_{n} < 0$. Such a behavior is typical if the system is not minimal phase with respect to the cost functional and has to be accounted for if the cost functional cannot be adapted appropriately. We point out that the conditions of Theorem \ref{Acceptable violations of relaxed Lyapunov inequality:thm:estimate} and Corollary \ref{Acceptable violations of relaxed Lyapunov inequality:cor:exit strategy} with respect to $\lim_{n \to \infty} s_n$, unlike conditions \eqref{Setup:thm:aposteriori:eq1} or \eqref{Reducing the optimization horizon using longer control horizons:prop:trajectory estimate:eq1}, cannot be checked at runtime. Still, while the maximal $\overline{\alpha}$ satisfying \eqref{Setup:thm:aposteriori:eq1} or \eqref{Reducing the optimization horizon using longer control horizons:prop:trajectory estimate:eq1} is locally a lower bound on the degree of suboptimality, cf. \eqref{Reducing the optimization horizon using longer control horizons:prop:trajectory estimate:eq1}, the knowledge of $s_n$ allows us to compute such a bound for a horizon of length $\sum_{i = 0}^{n} m_i$, cf. \eqref{Acceptable violations of relaxed Lyapunov inequality:cor:exit strategy:eq1}. Note that the latter bound uses the stage costs as weighting factors.

\begin{corollary}\label{Acceptable violations of relaxed Lyapunov inequality:cor:suboptimality degree}
	Consider a feedback law $\mu_{N}^{\cS}(\cdot; \cdot)$ with sequence $(m_{n})_{n \in \N}$, $m_{n} \in \{1,2, \ldots, N - 1\}$ for all $n \in \N$ as well as the corresponding closed loop trajectory $x_{n}$, $n \in \N_0$, of \eqref{Reducing the optimization horizon using longer control horizons:eq:closed loop} with initial value $x_0 \in \bX$ to be given. Furthermore, suppose $\overline{\alpha} \in (0, 1]$ to be fixed and $(s_{n})_{n \in \N_0}$ to be defined as in Theorem \ref{Acceptable violations of relaxed Lyapunov inequality:thm:estimate}. Then we have the following:\\
	(i)
	\begin{align}
		\label{Acceptable violations of relaxed Lyapunov inequality:cor:suboptimality degree:eq1}
		\{ \overline{\alpha} \mid \overline{\alpha} \mbox{ maximally satisfies \eqref{Acceptable violations of relaxed Lyapunov inequality:cor:exit strategy:eq1}} \} = \frac{V_{N}(x_0) - V_{N}(x_{n + 1})}{V_{N}(x_0) - V_{N}(x_{n + 1}) - s_{n}} \overline{\alpha}.
	\end{align}
	(ii) If $(s_{n})_{n \in \N_0}$ converges with $\lim_{n \to \infty} s_{n} = \theta$, then \eqref{Reducing the optimization horizon using longer control horizons:prop:trajectory estimate:eq2} holds with degree of suboptimality $\alpha = \overline{\alpha} V_{N}(x_0) / (V_{N}(x_0) - \theta)$.
\end{corollary}
\begin{proof}
	To prove (i) we first reformulate the definition of $s_{n}$ to obtain
	\begin{align}
		\label{Acceptable violations of relaxed Lyapunov inequality:cor:suboptimality degree:proof:eq1}
		\sum\limits_{i = 0}^{n} \sum\limits_{k = 0}^{m_{i} - 1} \l(x_{u_{N}}(k; x_{i}), u_{N}(k; x_{i})) = \frac{ V_{N}(x_{0}) - V_{N}(x_{n + 1}) - s_{n} }{ \overline{\alpha} }
	\end{align}
	where we used the equivalence of the open and closed loop control $u_{N}(\cdot; x_{i})$, $\mu_{N}^{\cS}(\cdot; x_{i})$ from \eqref{Reducing the optimization horizon using longer control horizons:eq:closed loop control} to replace $x_{u_{N}}(m_{i}; x_{i})$ by $x_{i + 1}$ for $i = 0, \ldots, n$. To obtain $\{ \overline{\alpha} \mid \overline{\alpha} \mbox{ maximally satisfies \eqref{Acceptable violations of relaxed Lyapunov inequality:cor:exit strategy:eq1}} \}$ we consider \eqref{Acceptable violations of relaxed Lyapunov inequality:cor:exit strategy:eq1} to hold as an equality and solve for $\overline{\alpha}$. Now, we can use \eqref{Acceptable violations of relaxed Lyapunov inequality:cor:suboptimality degree:proof:eq1} to substitute the resulting denominator which gives us \eqref{Acceptable violations of relaxed Lyapunov inequality:cor:suboptimality degree:eq1}.

	In order to obtain assertion (ii), we use the following fact shown in the proof of Theorem \ref{Acceptable violations of relaxed Lyapunov inequality:thm:estimate}: If $(s_{n})_{n \in \N_0}$ converges, then $V_{N}(x_{i}) \to 0$ as $i \to \infty$. Hence, taking $n$ to infinity in \eqref{Acceptable violations of relaxed Lyapunov inequality:cor:suboptimality degree:eq1} shows assertion (ii).
\end{proof}

Corollaries \ref{Acceptable violations of relaxed Lyapunov inequality:cor:exit strategy} and \ref{Acceptable violations of relaxed Lyapunov inequality:cor:suboptimality degree} give rise to a possible exit strategy in Step (Id): If $s_{n}$ remains positive, then asymptotic stability and the desired performance bound can be guaranteed if the control loop is closed. One possible implementation of an algorithm based on Corollary \ref{Acceptable violations of relaxed Lyapunov inequality:cor:exit strategy} is the following:

\begin{fshaded}
	\begin{algorithm}[Extension of Algorithm \ref{Reducing the optimization horizon using longer control horizons:alg:basic algorithm}]\label{Acceptable violations of relaxed Lyapunov inequality:alg:rho improvement}$\\ $
	Given state $x_{0} := x$, $n = 0$, list $\cS = (n)$, $N \in \mathbb{N}_{\geq 2}$, $\overline{\alpha} \in (0, 1]$ and $s_{n} = 0$
	\begin{itemize}
		\item[(I)] Set $j := 0$ and compute $u_{N}(\cdot; x_{n})$ and $V_{N}(x)$. Do
		\begin{itemize}
			\item[(a)] Set $j := j + 1$ and compute $V_{N}(x_{u_{N}}(j; x_{n}))$
			\item[(b)] Compute $\rho_{n}$ from \eqref{Acceptable violations of relaxed Lyapunov inequality:eq:rho} and set $\rho_{n}(j) := \rho_{n}$ 
			\item[(c)] If $\rho_{n}(j) \geq 0$ holds: Set $m_{n} = j$ and goto (II)
			\item[(d)] If $j = N - 1$:
			\begin{itemize}
				\item[] If $s_{n} + \max_{j \in \{1, \ldots, N - 1\}} \rho_{n}(j) < 0$: Print warning, set $m_{n} := 1$ and goto (II)
				\item[] Else: Set $m_{n} := \min\{ j^{*} \mid j^{*} = \argmin_{j \in \{1, \ldots, N - 1\}} - \rho_{n}(j) \}$, $\rho_{n} := \rho_{n}(m_n)$ and goto (II)
			\end{itemize}
		\end{itemize}
		while $\rho_{n} < 0$
		\item[(II)] For $j = 1, \ldots, m_{n}$ do
		\begin{itemize}
			\item[] Implement $\mu_{N}^{\cS}(j - 1; x_{n}) := u_{N}(j-1; x_{n})$
		\end{itemize}
		\item[(III)] Set $\cS := (\cS, \mbox{back}(\cS) + m_{n})$, $s_{n + 1} := s_{n} + \rho_{n}(m_{n})$, $x_{n + 1} := x_{\mu_{N}^{\cS}}(m_{n}; x_{n})$, $n := n + 1$ and goto (I)
	\end{itemize}
	\end{algorithm}
	\vspace*{-0.1cm}
\end{fshaded}

If $s_{n + 1} = s_{n} + \max_{j \in \{1, \ldots, N - 1\}} \rho_{n}(j) < 0$ the performance Estimate \eqref{Reducing the optimization horizon using longer control horizons:prop:trajectory estimate:eq2} cannot be guaranteed. Hence, even if $\lim_{n \to \infty} s_{n} \geq 0$ holds, the exit strategy cannot be used since this knowledge is not at hand at time instant $n$. Furthermore, we like to point out that Algorithm \ref{Acceptable violations of relaxed Lyapunov inequality:alg:rho improvement} coincides with Algorithm \ref{Reducing the optimization horizon using longer control horizons:alg:basic algorithm} for $n=0$, i.e., the first MPC step. Note that this is the only time instant at which we may increase the optimization horizon $N$ such that the presented stability proofs still hold. Hence, we can repeat Step (I) of the algorithm in order to ensure the desired relaxed Lyapunov inequality. In contrast to that, reducing $N$ may be done at runtime of the proposed algorithms.

\begin{remark}\label{Acceptable violations of relaxed Lyapunov inequality:rem:shortening}
	Using the definition of $\rho_i$ in \eqref{Acceptable violations of relaxed Lyapunov inequality:eq:rho} and $V_N(x) \geq V_{N - k}(x)$ for $k \in \{ 0, \ldots, N - 2 \}$ we obtain
	\begin{align*}
	\hat{\rho}_i := V_{N}(x_{i}) - V_{N - k}(x_{u_{N}}(m_{i}; x_{i})) - \overline{\alpha} \sum\limits_{k = 0}^{m_{i} - 1} \l(x_{u_{N}}(k; x_{i}), u_{N}(k; x_{i})) \geq \rho_i(\overline{\alpha}) = \rho_i
	\end{align*}
	for $k \in \{ 0, \ldots, N - 2 \}$. Hence, the telescope sum argument used in the proofs of Proposition \ref{Reducing the optimization horizon using longer control horizons:prop:trajectory estimate} and Theorem \ref{Acceptable violations of relaxed Lyapunov inequality:thm:estimate} still holds if the horizon length $N$ decreases monotonically along the closed loop.
\end{remark}

Remark \ref{Acceptable violations of relaxed Lyapunov inequality:rem:shortening} gives rise to the following strategy: First, a large optimization horizon is chosen to avoid the startup problem. Then, the horizon can be reduced gradually along the closed loop provided the relaxed Lyapunov inequality \eqref{Reducing the optimization horizon using longer control horizons:prop:trajectory estimate:eq1} holds for the reduced horizon for all future time instants. In context of Theorem \ref{Acceptable violations of relaxed Lyapunov inequality:thm:estimate} one has to guarantee that $s_n$ stays positive along the closed loop in order to reduce the optimization horizon. Note that an algorithm based on a quantity representing slack from proceeding steps has also been designed in \cite[Theorem 1]{G2010} in the context of varying optimization horizons.

\begin{remark}
	For $\overline{\alpha} > 0$ an additional exit strategy is the following: if $0 < \alpha < \overline{\alpha}$ holds for suboptimality index $\alpha$ computed from Corollary \ref{Acceptable violations of relaxed Lyapunov inequality:cor:suboptimality degree}(i), then stability or a certain performance bound may still be ensured. Additionally, Corollary \ref{Acceptable violations of relaxed Lyapunov inequality:cor:exit strategy} allows for checking whether the originally desired suboptimality estimate $\alpha \geq \overline{\alpha}$ is guaranteed again at a later time instant.
\end{remark}

In order to check whether the control loop can be closed more often without loosing stability or violating the lower bound on the degree of suboptimality $\overline{\alpha}$, Corollaries \ref{Acceptable violations of relaxed Lyapunov inequality:cor:exit strategy} and \ref{Acceptable violations of relaxed Lyapunov inequality:cor:suboptimality degree} are employed directly. Note that this is possible since the sequence $(s_{n})_{n \in \N_0}$ gives us an absolute value with respect to the desired decrease along the closed loop -- contrary to Propositions \ref{Reducing the optimization horizon using longer control horizons:prop:trajectory estimate} and \ref{Robustification by reducing the control horizon:prop:m2one estimate} which have to be interpreted in terms of the stage costs $\l(\cdot, \cdot)$. One possible implementation of the update check is the following:

\begin{fshaded}
	\begin{algorithm}[Modification of Algorithm \ref{Acceptable violations of relaxed Lyapunov inequality:alg:rho improvement}]\label{Acceptable violations of relaxed Lyapunov inequality:alg:rho improvement2}$\\ $ \vspace*{-0.5cm}
	\begin{itemize}
		\item[(II)] Set $\hat{n} := 1$ and $s := \mbox{back}(\cS)$. For $j = 1, \ldots, m_{n}$ do
		\begin{itemize}
			\item[(a)] Implement $\mu_{N}^{\cS}(j-1; x_{n}) := u_{N}(j-1; x_{n})$
			\item[(b)] Compute $u_{N}(\cdot; x_{u_{N}}(j; x_{n}))$, construct $\hat{u}_{N}(\cdot; x_{n})$ according to \eqref{Robustification by reducing the control horizon:eq:control update} and compute $V_{N}(x_{\hat{u}_{N}}(m_{n}; x_{n}))$
			\item[(c)] Compute $\rho_{n}$ from \eqref{Acceptable violations of relaxed Lyapunov inequality:eq:rho} with $u_{N}(\cdot; x_{n})$ replaced by $\hat{u}_{N}(\cdot; x(n))$ and set $\rho_{n}(j) := \rho_{n}$ 
			\item[(d)] If $s_{n + \hat{n} - 1} + \max_{j \in \{1, \ldots, N - 1\}} \rho_{n}(j) \geq 0$ holds:
			\begin{itemize}
				\item[] Update $u_{N}(\cdot; x_{n}) := \hat{u}_{N}(\cdot; x_{n})$.
				\item[] If $j < m_{n}$: Set $\cS := ( \cS, s + j )$, $s_{n + \hat{n}} := s_{n + \hat{n} - 1} + \rho_{n}(j)$, $x_{n + \hat{n}} := x_{\mu_{N}^{\cS}}(j; x_{n})$  and $\hat{n} = \hat{n} + 1$. 
			\end{itemize}
 		\end{itemize}
		\item[(III)] Set $\cS := (\cS, s + m_{n})$, $s_{n + \hat{n}} := s_{n + \hat{n} - 1} + \rho_{n}(m_{n})$, $x_{n + \hat{n}} := x_{\mu_{N}^{\cS}}(m_{n}; x_{n})$, $n := n + \hat{n}$ and goto (I)
	\end{itemize}
	\end{algorithm}
	\vspace*{-0.1cm}
\end{fshaded}

\begin{example}\label{Acceptable violations of relaxed Lyapunov inequality:ex}
	Consider Example \ref{Reducing the optimization horizon using longer control horizons:ex1} one last time. As we have seen in Example \ref{Reducing the optimization horizon using longer control horizons:ex2} and \ref{Robustification by reducing the control horizon:ex} stability of the closed loop can be shown for initial values $x_0 \in \mathcal{X}$ by means of Propositon \ref{Reducing the optimization horizon using longer control horizons:prop:trajectory estimate} for $m_n \geq 1$ and by Corollary \ref{Robustification by reducing the control horizon:cor:stability} for $m_n = 1$. In Example \ref{Robustification by reducing the control horizon:ex} it has also been shown that one cannot guarantee the lower performance bound $\overline{\alpha} = 0.01$ by Theorem \ref{Robustification by reducing the control horizon:thm:stability}. Yet, we would expect a better performance of the Riccati based feedback law. And indeed, using Algorithm \ref{Acceptable violations of relaxed Lyapunov inequality:alg:rho improvement2} together with Corollary \ref{Acceptable violations of relaxed Lyapunov inequality:cor:suboptimality degree} we obtain $\alpha = 0.52307$ for standard MPC ($m_n = 1$) with horizon length $N = 3$. In Figure \ref{Acceptable violations of relaxed Lyapunov inequality:fig:alpha development} the values of $\alpha$ resulting from Proposition \ref{Reducing the optimization horizon using longer control horizons:prop:trajectory estimate} and Corollary \ref{Acceptable violations of relaxed Lyapunov inequality:cor:suboptimality degree} are displayed for different horizons $N$ showing the improvement of Corollary \ref{Acceptable violations of relaxed Lyapunov inequality:cor:suboptimality degree}. For reasons of completeness, we also displayed the performance results from \cite[Section 6]{NP1997}. Note that while the latter hold for the entire state space $\X = \R^2$, our results are only exact up to discretization accuracy. Still, the improvement of suboptimality bounds is significant and allows us to reduce the optimization horizon from $N = 5$ as shown in \cite{NP1997} to $N = 3$. \\
	Unfortunately, we observe $s_{n + 1} = s_{n} + \max_{j \in \{1, \ldots, N - 1\}} \rho_{n}(j) < 0$ along the closed loop for exactly the same initial values for which condition $\alpha \geq \overline{\alpha}$ in Step (Ic) of Algorithm \ref{Reducing the optimization horizon using longer control horizons:alg:basic algorithm} is not satisfied for at least one iterate $n$, cf. Figure \ref{Acceptable violations of relaxed Lyapunov inequality:fig:alpha development}(right).
\end{example}
\begin{figure}[!ht]
	\begin{center}
		\includegraphics[width=5cm]{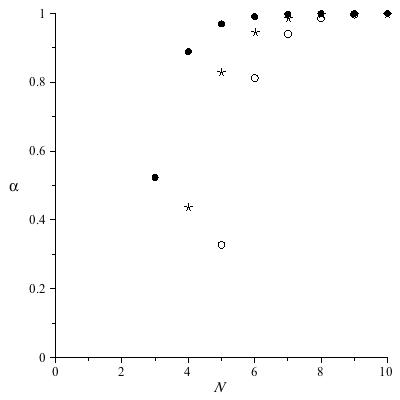}
		\hspace*{1cm}
		\includegraphics[width=5cm]{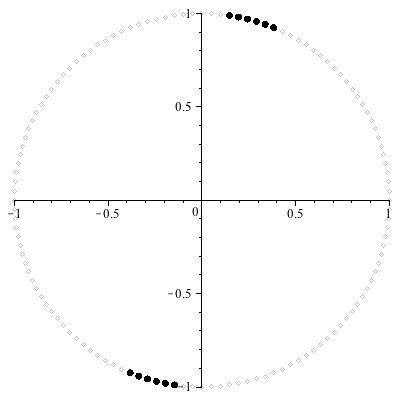}
		\caption{Left: Illustration suboptimality bounds from \cite[Section 6]{NP1997} ($\circ$), $\overline{\alpha}$ from Proposition \ref{Reducing the optimization horizon using longer control horizons:prop:trajectory estimate} ($\star$) and $\alpha$ from Corollary \ref{Acceptable violations of relaxed Lyapunov inequality:cor:exit strategy} ($\bullet$) for increasing horizon length $N$. Right: Illustration of initial values for which $s_{n + 1} = s_{n} + \max_{j \in \{1, \ldots, N - 1\}} \rho_{n}(j) < 0$ holds for at least one iterate $n$ ($\bullet$)}
		\label{Acceptable violations of relaxed Lyapunov inequality:fig:alpha development}
	\end{center}
\end{figure}

\begin{remark}
	The proposed algorithms and theoretic results in this paper are designed in a trajectory based manner. Therefore, these methods are not suited in order to ensure asymptotic stability or a desired suboptimality degree for a set of initial values $\mathcal{X} \subset \bX$ --- in contrast to the approaches presented in \cite{G2009, NP1997}. Yet, incorporating the conditions presented in Proposition \ref{Robustification by reducing the control horizon:prop:m2one estimate} or Corollary \ref{Acceptable violations of relaxed Lyapunov inequality:cor:exit strategy} in the methodology proposed in \cite[Section 4]{G2009} is possible. This topic will be subject to future research.
\end{remark}

\section{Conclusions}
\label{Section:Conclusions}

An algorithm based approach has been presented which ensures stability of the MPC closed loop without terminal constraints and/or costs. In particular, the proposed methodology allows for deducing stability and performance bounds for comparatively small prediction horizons. These results are based on structural properties of a relaxed Lyapunov inequality for the open loop which are computed a priori. Conditions which guarantee that this Lyapunov inequality is maintained despite closing the control loop at additional time instants have been derived in order to robustify the outcome of the corresponding algorithm. Last, a further improvement has been achieved by incorporating an accumulated quantity in the presented algorithms which reflects previous decrease in terms of the value function of the MPC problem. Doing so yields an exit strategy which often resolves problems occuring within our basic algorithms if the prediction horizon is chosen too small. Furthermore, enhanced performance estimates are obtained.

\providecommand{\WileyBibTextsc}{}
\let\textsc\WileyBibTextsc
\providecommand{\othercit}{}
\providecommand{\jr}[1]{#1}
\providecommand{\etal}{~et~al.}


\end{document}